\documentclass[11pt]{amsart}
\usepackage{a4wide}
\usepackage{amsmath, amsfonts, amssymb}

\def\id{\textrm{Id}}

\newcommand{\D}{\ensuremath{\mathbb{D}}} 
\newcommand{\R}{\ensuremath{\mathbb{R}}} 

\newcommand{\longto}{\longrightarrow}

\newtheorem{theo}{Theorem}
\newtheorem{prop}[theo]{Proposition}

\newtheorem{rem}[theo]{Remark}

\def\D{\displaystyle}

\newcommand{\var}{\varepsilon}

\def\te{\theta}

\def\longto{\longrightarrow}
\def\D{\displaystyle}

\newcommand{\be}{\begin{equation}}
\newcommand{\ee}{\end{equation}}

\def\benu{\begin{enumerate}}
\def\eenu{\end{enumerate}}

\def\pbd#1{probability density w.r.t. \ensuremath{#1}}

\title{The geometry of  Euclidean convolution inequalities and entropy}
\author{Dario Cordero-Erausquin and Michel Ledoux}

\begin{document}

\maketitle

\begin{abstract}
The goal of this note is to show that some convolution type inequalities from Harmonic Analysis and Information Theory, such as Young's convolution inequality (with sharp constant), Nelson's hypercontractivity of the Hermite semi-group or Shannon's inequality, can be reduced to a simple geometric study of frames of $\R^2$.
We shall derive directly entropic inequalities, which were recently proved to be dual to the Brascamp-Lieb convolution type inequalities. 
\end{abstract}

\bigskip

\section{Introduction}

The topic of Brascamp-Lieb and convolution type inequalities was recently renewed by Carlen, Lieb and Loss~\cite{CLL}  who proposed a semi-group or heat flow approach to these inequalities. (Soon after  Bennett, Carbery, Christ and Tao~\cite{BCCT1} independently  gave a semi-group approach to multidimensional Brascamp-Lieb inequalities.)
Carlen, Lieb and Loss also obtained new inequalities on the sphere, and in particular a  subadditivity of the entropy inequality. It was noted in~\cite{BCM}  that this inequality can be proved using geometric properties of the Fisher information of the marginal distributions. Pushing forward these investigations, Carlen and Cordero-Erausquin~\cite{CC} derived a similar geometric treatment for general subbaditivity of the entropy inequalities on Euclidean space, and proved that these inequalities are dual to the Brascamp-Lieb inequalities. The semi-group approach was recently carried out in the unifying framework of abstract Markov semi-groups in~\cite{BCML}.

The abstract geometric argument in $\R^n$ is particularly simple and for self-consistency, we  describe it in details in this introduction.

In the sequel, $\mu$ will stand either for the Lebesgue measure on $\R$ or for the standard Gaussian probability measure $\gamma$ on $\R$ (with density $(2\pi)^{-1/2}e^{-|t|^2/2}$). Consequently, $\mu_n := \mu^{\otimes n}$ will stand for the Lebesgue measure or the standard Gaussian measure $\gamma_n$ on $\R^n$.  It is convenient to treat these two cases in parallel although it is possible to derive formally one from another.

Say that $f$ is a probability density with respect to (w.r.t.) $\mu_n$
if $f:\R^n\to \R^+$ is such that $\int\! f\, d\mu_n=1$. 
Given a random vector $X\in \R^n$ with $f$ as \pbd{\mu_n} (a relation written below as  $X\sim f\, d\mu_n$), its \emph{entropy} w.r.t $\mu_n$ is defined (whenever it makes sense) by
$$S_{\mu_n}(X):=S_{\mu_n}(f):= \int_{\R^n} f \log f\, d{\mu_n}.$$
In the case $\mu_n$ is the Lebesgue measure we shall use the notation $S(X)=S(f)=\int f\log f $. All along the paper, it will be implicitly assumed in all statements that we consider only densities and random vectors with well defined and finite $\mu_n$-entropy.

If $f$ is \pbd{\mu_n} on $\R^n$ and $a\in \R^n$ is a fixed non-zero vector, denote by  $f_{(a)}$ the marginal \pbd{\mu} on $\R$, i.e. $f_{(a)}\, d\mu$  is the image of $f\, d\mu$ under the map $x\to a\cdot x$.
Thus, $f_{(a)}$ is characterized by the requirement that
\begin{equation}\label{def:marginal}
\int_{\R^n} \phi (x\cdot a) f(x) \,d\mu_n(x) = \int_{\R} \phi (t) f_{(a)}(t) \,d\mu_1(t)  
\end{equation}
for every bounded measurable $\phi:\R\to\R$.
Equivalently,  if $X\sim f\, d\mu_n$, then $f_{(a)}\, d\mu_1$ is the density of $X\cdot a$, that is $X\cdot a\sim f_{(a)}\, d\mu_1$. Thus,
$S_{\mu}(X\cdot a)=S_{\mu}(f_{(a)})= \int_{\R} f_{(a)}(t)\log f_{(a)}(t)\, d\mu(t) $.
The classical subadditivity of entropy (usually stated with the Lebesgue measure) indicates that for an orthonormal basis $(u_1, \ldots, u_n)$ of $\R^n$ and a random vector $X$,
\begin{equation}\label{eq:classicalsub}
\sum_{i=1}^n S_{\mu_1}(X\cdot u_i) \le S_{\mu_n} (X).
\end{equation}
The relation between subadditivity inequalities and Brascamp-Lieb inequalities is summarized in the following proposition.
 \begin{prop}[\cite{CC}]  \label{prop:dual}
 For non-zero vectors $a_1,\ldots, a_m \in \R^n$, $c_1\, \ldots, c_m\ge 0$ and $D\in \R$, the following assertions are equivalent:
\begin{enumerate}
\item For every $f_1,\ldots, f_m :\R\to\R^+$, 
$\D \int_{\R^n}\prod_{i=1}^m f_i (x\cdot a_i)^{c_i} \, d\mu_n(x) \; \le e^D \; \prod_{i=1}^m \left(\int_{\R} f_i \, d\mu_1\right)^{c_i}.$
\item For every random vector $X\in \R^n$,
$\D  \sum_{i=1}^m c_i \, S_{\mu}(X\cdot a_i) \; \le\;  S_{\mu_n}(X)  \; + \;D .$
\end{enumerate}
\end{prop}

We also have a complete equivalence between the equality cases. This result is easy to prove (actually it holds in much more general settings): it formally relies  on the fact that the the Legendre transform of the entropy functional is the functional $V\to \log \int e^{V}\, d\mu_n$, that is
\begin{equation*}\label{dualentropy}
\log\int e^V d\mu_n= \sup_f \left \{ \int f V\, d\mu_n - S_{\mu_n}(f) \right \}\textrm{ and }
S_{\mu_n}(f) = \sup_V \left\{ \int f V \, d\mu_n - \log\int e^V \, d\mu_n\right\},
\end{equation*}
and on how this combines with~\eqref{def:marginal}. The dual of the subadditivity
inequality~\eqref{eq:classicalsub} is nothing else but Fubini's theorem.

It is possible to consider more general geometric situations than the one of an orthonormal basis~\eqref{eq:classicalsub}. Of particular interest is the case of a decomposition of the identity, as put forward by Ball in the context of  Brascamp-Lieb inequalities (see e.g.~\cite{Ball}). Given \emph{unit} vectors $u_1, \ldots , u_m$ in the Euclidean space $\R^n$ and real numbers $c_1, \ldots, c_m >0$, we say that they decompose the identity if
\begin{equation}\label{decomp}
\sum_{i=1}^m c_i \, u_i \otimes u_i = \id_{\R^n}
\end{equation}
where $u_i \otimes u_i $ stands for the orthogonal projection in the {direction} $u_i$. Note that necessarily
$c_i\le 1$  and
\begin{equation}\label{comp}
\sum_{i=1}^m c_i = n.
\end{equation}

It is easy to derive sharp subadditivity entropy inequalities using~\eqref{decomp} because such decompositions combine nicely with the Fisher information, as noted in~\cite{BCM}. The point is that the Fisher information has an $L^2$ structure which allows for geometric operations such as  projections (or equivalently,  conditional expectation). A random vector $X\in \R^n$ with  $f$ as \pbd{\mu_n} is said to have finite $\mu_n$-Fisher information  if the following quantity is well defined and finite:
$$I_{\mu_n}(X):=I_{\mu_n}(f):=\int_{\R^n}\frac{|\nabla f|^2}f \, d\mu_n.$$
It follows from the Cauchy-Schwarz inequality (see~\cite{carlen, CC}) that for a unit vector $u\in \R^n$,
\begin{equation}\label{supadd1}
I_{\mu_1}(f_{(u)})  = I_{\mu_1}(X\cdot u) \le   \int_{\R^n} \frac{(\nabla f \cdot u)^2}f\, d\mu_n\, ,
\end{equation}
with equality if and only if $X\cdot u$ and $X- (X\cdot u)u$ are independent. If we are given a decomposition of the identity, then, rewriting~\eqref{decomp}
in the form
\be\label{decomp0}
\forall v \in \R^n , \quad \sum_{i=1}^m c_i \, (v\cdot u_i)^2 = |v|^2 ,
\ee
we immediately get from~\eqref{supadd1} that for any random vector $X$ with finite Fisher information,
\begin{equation}\label{infosub}
 \sum_{i=1}^m c_i \, I_{\mu}(X\cdot u_i) \le I_{\mu_n}(X) .
 \end{equation}
 In order to get an inequality for entropy, we integrate along the suitable semi-group. Let $L$ stand for the differential operator $Lf= \Delta f$ (Laplacian) when $\mu_n$ is the Lebesgue measure, and $Lf=\Delta f - x\cdot\nabla f$ in the case $\mu_n=\gamma_n$. Let $P_t = e^{tL}$ be the corresponding heat semi-group and Ornstein-Uhlenbeck semi-group, which admit as invariant measure the Lebesgue and the Gaussian measure respectively. If $X$ is a random vector with density $f$ with respect to $\mu_n$ and finite $\mu_n$-entropy, then $e^{tL}f$ is a smooth probability density with respect to $\mu_n$,  with finite $\mu_n$-Fisher information, and
$$\frac{{\rm d}}{{\rm d}t} \, S_{\mu_n}(e^{tL}f) = -I_{\mu_n}(e^{tL}f).$$
Moreover, $L$ (and thus $e^{tL}$) has the property that it preserves the algebra of functions of the form $f(x)=g(x\cdot u)$, a property also used in the proof of~\eqref{supadd1}. This ensures the following crucial property,
namely, for every $t\ge 0$,
\be\label{stability}
(e^{tL}f)_{(u)}  = e^{tL} (f_{(u)})\quad \textrm{ on } \R,
\ee
where we used the same notation for the one-dimensional and $n$-dimensional semi-groups. The heat and
Ornstein-Uhlenbeck semi-groups are the two most important diffusion semi-groups sharing property~\eqref{stability}, and this explains the particular role played here by the Lebesgue and Gaussian measures.

Now, integration of~\eqref{infosub} along the semi-group $e^{tL}$ leads to the inequality
$\sum_{i=1}^m c_i\, S_{\mu}(X\cdot u_i) \le S_{\mu_n}(X).$
From the cases of equality in~\eqref{supadd1}  we get that equality holds if and only if  for each $i\le m$,
$X\cdot u_i$ and $X - (X\cdot u_i)u_i$ are independent (a property which is preserved along the semi-group). Under mild conditions on the vectors $u_i$, it is easily seen that this can happen only when $X$ is a Gaussian vector (see~\cite{CC} for details). The previous discussion is summarized in the next proposition, established in~\cite{CC}.

\begin{prop} \label{special}
Consider a decomposition of the identity~\eqref{decomp} in  $\R^n$  Then, for all random vectors $X\in \R^n$,
\begin{equation}\label{entropysub2}
\sum_{i=1}^m c_i\,  S_{\mu}(X\cdot u_i) \le  S_{\mu_n}(X)  .
\end{equation}
Furthermore, under the condition that  no two of the unit vectors $\{u_i\}$ are linearly dependent, and that
if any one vector $u_i$ is removed from $\{u_1,\dots,u_m\}$, the remaining vectors still span
$\R^n$, equality holds in~\eqref{entropysub2}  if and only if $X$ is a Gaussian random variable whose covariance is a multiple of the identity.
\end{prop}

In view of the duality given by Proposition~\ref{prop:dual}, we recover from the previous inequality Ball's form of the Brascamp-Lieb inequality: for every $f_1,\ldots, f_m :\R\to\R^+$, 
\be\label{BL}
\D \int_{\R^n}\prod_{i=1}^m f_i (x\cdot u_i)^{c_i} \, d\mu_n(x) \; \le \; \prod_{i=1}^m \left(\int_{\R} f_i \, d\mu\right)^{c_i}.
\ee
Moreover (under the same hypothesis on the $u_i$'s as in the previous proposition), equality holds if and only if 
 the measures $f_i\, d\mu_1$ are Gaussian with the same covariance: $f_i(t)\, d\mu(t) = \lambda_i e^{ -\alpha^2 (t-v\cdot a_i)^2 }\, dt$ with $\lambda_i>0$, $\alpha\in \R^\ast$ and $v\in \R^n$ ($v=0$ if we restrict to centered functions). Note that using~\eqref{decomp0} it is also possible to pass from inequalities for the Lebesgue measure to inequalities for the standard Gaussian measure (and \emph{vice versa}) by the correspondance
 $f_i \longleftrightarrow f_i(t) \, e^{-|t|^2/2}$.

\bigskip

The goal of this note is to prove the efficiency of the theoretical aforementioned approach in some meaningful situations. More precisely, we aim at understanding convolution and information theoretic inequalities, such as the sharp Young convolution inequality, Nelson's hypercontractivity of the Hermite semi-group, or Shannon's inequality, as functional forms of some particular decompositions of the identity of $\R^2$. 
To do so, we will study, in the next section \S2, the decompositions of the identity of $\R^2$ by three vectors. 
We give a complete description of the relation between the coefficients $c_i$ and the vectors $u_i$ in this case. As a consequence, we obtain the following general inequality which can be viewed as the functional form of Proposition~\ref{theo:decomp2} below.

\begin{theo}\label{maintheo}
Let $p_1, p_2, p_3 >1$ be such that
$$\frac1{p_1} +\frac1{p_2} + \frac1{p_3} = 2,$$
and $\theta_2, \theta_3 \in ]0, \pi[$ be defined by
$$(\cos(\theta_2),\, \sin(\theta_2)) \!=\! \bigg (\sqrt{(p_1-1)(p_2-1)} , \, \sqrt{\frac{p_1\, p_2 \, (p_3-1)}{p_3}}\bigg)$$
and
$$ (\cos(\theta_3), \, \sin(\theta_3))\!=\! \bigg(-\sqrt{(p_1-1)(p_3-1)} , \, \sqrt{\frac{p_1\, p_3\, (p_2-1)}{p_2}}\bigg).$$
Then, for every random variables $X, Y\in \R$, 
 \begin{equation}\label{sub0}
\frac1{p_1} \, S_{\mu}\big(X\big) + \frac1{p_2}\,  S_{\mu}\big(\cos(\te_2) X + \sin(\te_2) Y\big)  +  \frac1{p_3} \, S_{\mu}\big(\cos(\te_3) X + \sin(\te_3) Y\big)  \le S_{\mu_2}\big(X,Y\big).
\end{equation}
with equality if and only if $X$ and $Y$ are independent identically distributed Gaussian variables.

Equivalently, for every functions $g\in L^{p_2}(\mu)$ and $ h\in L^{p_3}(\mu)$, setting $p_1'=\frac{p_1}{p_1-1}$,
\be\label{int0}
 \left\|\int  g(\cos(\te_2) x + \sin(\te_2) y)\, h(\cos(\te_3) x + \sin(\te_3) y) \, d\mu(y) \right\|_{L^{p'_1}(d\mu(x))} \le \|g\|_{L^{p_2}(\mu)}\, \|h\|_{L^{p_3}(\mu)}
 \ee
with equality if and only if $f$ and $g$ are of constant sign with $|g(t)|^{p_2}d\mu(t)=K_2\, e^{-\lambda (t-a_2)^2} dt$ and $|h(t)|^{p_3}d\mu(t)=K_3\, e^{-\lambda (t-a_3)^2} dt$ for $ a_1, a_2 \in \R$ and $K_2, K_3,\lambda\ge0$.
\end{theo}

We apply this result in section \S3 to the determination of the sharp constant in the Young convolution inequality. For this, we will work directly with the entropy and exploit the following  simple but useful invariance of the entropy (in the case of Lebesgue measure) under linear transformation: for a random vector $X\in\R^n$ and an invertible linear operator $A$ on $\R^n$,
\begin{equation} \label{entropyinv}
S(A X)= S(X) - \log(|\det(A)|) .
\end{equation}
In section \S4 we derive the Shannon inequality from a limit of decompositions of the identity. The hope is to
shep in this way new light on the connection between subadditivity of entropy and Shannon's inequality.
We study similarly Gaussian inequalities in section \S5 and derive in particular hypercontractity of the Hermite semi-group
(and the associated logarithmic Sobolev inequality).
For simplicity, we consider here only one-dimensional inequalities. In the last section \S6 we briefly explain how to extend word by word the approach to multi-dimensional situations.


\medskip

\section{Decomposition of the identity of $\R^2$}

As announced, we investigate here decompositions of the identity of $\R^2$.  A decomposition of the identity with only two vectors $u_1$ and $u_2$ holds  if and only if these two vectors form an orthonormal basis and $c_1=c_2=1$.  In order to get something of interest, consider the case of three distinct unit vectors in $\R^2$, $u_1$, $u_2$ and $u_3$. Note that $u_i\otimes u_i = (-u_i)\otimes(-u_i)$ so here and in the sequel, `distinct' really means that the directions $\R u_i$ are distinct. The first question we address is the following: if the directions are given, can we find
{\it positive} numbers $c_1,c_2,c_3$ such
that the decomposition of the identity
\begin{equation}\label{eq:decomp}
c_1 \, u_1 \otimes u_1 \;  + \;   c_2 \, u_2 \otimes u_2  \; + \;   c_3 \, u_3 \otimes u_3  \; =\;  \id_{\R^2}
\end{equation}
holds? The answer is yes  provided the vectors are `well enough' distributed in space.

\begin{prop}\label{theo:decomp}
Let $\R u_1$, $\R u_2$ and $\R u_3$ be three distinct directions of $\R^2$. There exists three positive numbers $c_1, c_2$ and $c_3$ such that the decomposition of the identity~\eqref{eq:decomp} holds if and only if the three geometric angles given by the six angular sectors defined by these directions are all strictly smaller than $\frac{\pi}2$.
The $c_i$'s are then given by 
\begin{equation}\label{c}
c_i = \frac{\cos(\te_j-\te_k)}{\sin(\te_j-\te_i)\sin(\te_k-\te_i)} = 1 - \cot(\te_i-\te_j) \cot(\te_k- \te_i) 
\end{equation}
for $(i,j,k)$ a permutation of $(1,2,3)$ and $u_i = \big(\cos(\te_i) , \sin(\te_i)\big)$ for  $i=1,2,3$.
\end{prop}

\begin{proof}
If  two of the vectors are orthogonal, say $u_1\cdot  u_2 =0$, and if~\eqref{eq:decomp} holds, then 
$0= u_1\cdot u_2 = 0 + 0 + c_3 \, (u_3\cdot u_1) \, (u_3\cdot u_2)$
and therefore, if $c_3>0$ then $u_3\cdot u_1 = 0$ or $u_3\cdot u_2 = 0$. But this implies $u_3=\pm u_2$ or $u_3= \pm u_1$, which is excluded. A genuine three vector situation cannot contain a two vector situation (which is equivalent to an orthonormal basis).

By the assumption, the projection $u_1\otimes u_1$, $u_2\otimes u_2$ and $u_3\otimes u_3$ span
${\mathcal S}_2$, the $3$-dimensional space of symmetric operators on $\R^2$. Therefore the linear operator 
$$(c_1, c_2, c_3) \longto c_1 u_1 \otimes u_1 +  c_2 u_2 \otimes u_2   +  c_3 \, u_3 \otimes u_3  $$
is an isomorphism from $\R^3$ onto ${\mathcal S}_2$.
Write $u_i = \big(\cos(\te_i) \, , \, \sin(\te_i)\big) $ for $i=1,2,3$, 
where all coordinate computations are done in the canonical orthonormal basis of $\R^2$; note that $\te_i-\te_j \neq 0 [\frac{\pi}2]$. Using that
$$u_i\otimes u_i = \left(
\begin{array}{cc}
 \cos(\te_i)^2 &\cos(\te_i)\sin(\te_i) \\
\cos(\te_i)\sin(\te_i) &\sin(\te_i)^2
\end{array}
\right), 
$$
it is readily checked that the unique solution in $\R^3$ of~\eqref{eq:decomp} is given by
$$\textstyle
c_1 = \frac{\cos(\te_2-\te_3)}{\sin(\te_2-\te_1)\sin(\te_3-\te_1)}\, , \quad 
 c_ 2= \frac{\cos(\te_3-\te_1)}{\sin(\te_3-\te_2)\sin(\te_1-\te_2)} \, , \quad
c_3 = \frac{\cos(\te_1-\te_2)}{\sin(\te_1-\te_3)\sin(\te_2-\te_3)}\,  .
$$
It remains to identify when this gives a solution to our problem, i.e. when $c_1, c_2, c_3 >0$. 

All the quantities in the previous equation  remain unchanged if we replace some $\te_i$ by $\te_i+\pi$, which is consistent with the fact that we have been working with directions only. And of course, they also remain inchange by rotations, i.e. by $\theta_i \to \theta_i +\alpha$ for $i=1,2,3$ and $\alpha\in \R$. Therefore, up to a relabeling of the directions,  we can assume that
$0=\te_1 < \te_2 < \te_3 <\pi$.
Then the $c_i's$ are positive if and only if
$$\te_3 - \te_2 < \pi/2 , \quad  \te_3 -0   > \pi/2 , \quad \te_2 -0 <\pi/2 .$$
Rewriting the second condition as
$\pi - \te_3 < \pi/2,$
we get the announced condition on the three angular sectors $\te_2$, $\te_3-\te_2$, $\pi-\te_3$.
\end{proof}

We now investigate the  converse procedure. Given three numbers $c_1,c_2, c_3 \in (0,1)$ such that
\begin{equation}\label{eq:c}
c_1+c_2+c_3 = 2,
\end{equation}
we would like to know whether is possible to find directions $u_i= (\cos(\theta_i) , \sin(\te_i))$ for which  the decomposition of the identity~\eqref{eq:decomp} holds. The answer is yes, and the construction is unique up an isometry of $\R^2$ (which clearly preserves decompositions of the identity).

\begin{prop}\label{theo:decomp2}
For given $c_1,c_2, c_3 \in (0,1)$ satisfying~\eqref{eq:c}, there exists a triple of directions $\R u_1$, $\R u_2$ and $\R u_3$ unique up to isometries such that the decomposition of the identity~\eqref{eq:decomp} holds (the solutions $u_i=(\cos(\theta_i), \sin(\theta_i))$ are given by equation~\eqref{cond0} below). More explicitly, all solutions are obtained by performing an isometry on 
\begin{equation}\label{eq:decomp2}
u_1 \!=\! (1,0) , \ u_2 \!=\! \left(\sqrt{\frac{(1\!-\! c_1)(1\!-\! c_2)}{c_1\, c_2}} , \, \sqrt{\frac{1-c_3}{c_1\, c_2}}\right),\  u_3 \!=\! \left(-\sqrt{\frac{(1\!-\! c_1)(1\!-\! c_3)}{c_1\, c_3}} , \, \sqrt{\frac{1-c_2}{c_1\, c_3}}\right).
\end{equation}
\end{prop}

\begin{proof}
Inverting formally~\eqref{c} we get
\begin{gather}\textstyle
\cot(\te_2-\te_3) = \varepsilon \sqrt{\frac{(1-c_2)(1-c_3)}{1-c_1}}   \, , \quad
\cot(\te_3-\te_1)=  \varepsilon\sqrt{\frac{(1-c_3)(1-c_1)}{1-c_2}} \, , \nonumber \\ \textstyle
 \cot(\te_1-\te_2)=  \varepsilon \sqrt{\frac{(1-c_1)(1-c_2)}{1-c_3}}  \, \label{cond0}
 \end{gather}
and $ \varepsilon=\pm 1$. This uniquely determines the directions $\R u_i$ up to isometries. To check this,
first perform a rotation ensuring that
$\te_1 =0$  and $\te_2, \te_3 \in (0,\pi)$.
We still have an invariance by symmetry with respect to the coordinate axis $x=0$, which corresponds to the sign of $\varepsilon$. Thus, without loss of generality, we can impose
\begin{equation}\label{cond1}
0=\te_1 <  \te_2 < \te_3 <\pi .
\end{equation}
The last two equalities in~\eqref{cond0} give that $\cot(\te_2)$ and $\cot(\te_3)$ are of opposite sign, and thus, by~\eqref{cond1},
$\varepsilon = -1$,   $\theta_2 \in (0,\pi/2)$ and  $\te_3 \in (\pi/2 , \pi)$,  these angles being uniquely determined by
\begin{equation}\label{cond2}
\textstyle \cot(\theta_2) =   \sqrt{\frac{(1-c_1)(1-c_2)}{1-c_3}} \quad \textrm{ and } \quad
\cot(\theta_3)= -\sqrt{\frac{(1-c_3)(1-c_1)}{1-c_2}}  \, .
\end{equation}
Note that the first equality of~\eqref{cond0} ensures that $\te_3-\te_2\in (0,\pi/2)$, which  is consistent with the condition in the previous theorem. So long for unicity. It remains now  to check that we indeed get a solution. Equivalently, back to the situation~\eqref{cond1}-\eqref{cond2}, we need to check that once the last two equalities from~\eqref{cond0} are used to uniquely determine the angles $\theta_2$ and $\theta_3$, the first equality  of~\eqref{cond0} is then automatically verified.
This is indeed the case since
$$\textstyle
 \sqrt{\frac{1-c_1}{(1-c_2)(1-c_3)}}  \times \cot(\te_3-\te_2) =\sqrt{\frac{1-c_1}{(1-c_2)(1-c_3)}}  \times \frac{\cot(\te_2)\cot(\te_3)+1}{\cot(\te_2)- \cot(\te_3)}=\frac{c_1}{(1-c_2) + (1-c_3)} \, .
$$
Therefore the compatibility condition~\eqref{eq:c} yields the desired equation.  
 Finally, the solution~\eqref{cond1}-\eqref{cond2} rewrites as~\eqref{eq:decomp2} in coordinates
 and the proof is complete.
\end{proof}

We can now derive the main Theorem stated in the introduction.

\begin{proof}[Proof of Theorem~\ref{maintheo}]
Introducing, for $i=1,2,3$, $c_i =  \frac1{p_i}$,
the unit vectors $u_i=(\cos(\theta_i), \sin(\theta_i))$ of the previous proposition can be rewritten as $u_1 = (1,0)$, 
$$\textstyle
\ u_2 \!=\! \left(\sqrt{(p_1-1)(p_2-1)} , \, \sqrt{\frac{p_1\, p_2 \, (p_3-1)}{p_3}}\right),\  u_3 \!=\! \left(-\sqrt{(p_1-1)(p_3-1)} , \, \sqrt{\frac{p_1\, p_3\, (p_2-1)}{p_2}}\right).
$$
Then the result for entropy follows from Proposition~\ref{special}. Next note that it is enough to prove~\eqref{int0} in the case of nonnegative functions, and therefore the integral inequality to be proven is 
$$ \iint f(x)\,  g(\cos(\te_2) x + \sin(\te_2) y)\, h(\cos(\te_3) x + \sin(\te_3) y) \, d\mu(x)d\mu(y)  \le \|f\|_{L^{p_1}}\, \|g\|_{L^{p_2}(\mu)}\, \|h\|_{L^{p_3}}(\mu)$$
for $f,g,h:\R\to \R^+$. But this inequality holds as dual of inequality~\eqref{sub0}
by virtue of Proposition~\ref{prop:dual}. The cases of equality follow from the general considerations given in the introduction.
\end{proof}


\medskip

\section{Sharp Young's convolution inequality}
Here we work with the Lebesgue measure.
Let $p,q,r>1$ be such that
\be\label{coeffYoung}
\frac1p + \frac1q = 1+ \frac1r \,  ,
\ee
which can be rewritten as
$$ \frac1{r'} + \frac1p + \frac1q  = 2.$$
Apply then Theorem~\ref{maintheo} with  $p_1 = {r'}\, ,\  p_2 = p \, , \ p_3 = q$.
The angles $\te_2, \te_3$, or equivalently the unit vectors  $u_i=(\cos(\te_i), \sin(\te_i))$, $i=1,2,3$, given by Theorem~\ref{maintheo} are 
\be\label{condYoung}
u_1 \!=\! (1,0) , \ u_2 \!=\! \left(\sqrt{(r'-1)(p-1)} , \, \sqrt{\frac{r'\, p}{q'}}\right),\  u_3 \!=\! \left(-\sqrt{(r'-1)(q-1)} , \, \sqrt{\frac{r'\, q}{p'}}\right).
\ee
Therefore, for any random vector $(W,Z)\in \R^2$ (with finite entropy),
\begin{equation}\label{sub1}
\frac1{r'} \, S\big(W\big) + \frac1p\,  S\big(\cos(\te_2) W + \sin(\te_2) Z\big)  +  \frac1q \, S\big(\cos(\te_3) W + \sin(\te_3) Z\big)  \le S\big(W,Z\big).
\end{equation}
We would like to have as random variables in the left-hand side multiples of $W$, $W-Z$ 
and $Z$, respectively. To this task, first perform a linear transformation leaving $W$ invariant so that the last variable is a multiple of $Z$, and then a diagonal linear operator so that the second one is a multiple of $W-Z$. Readily, perform the linear transformation $(X,Y)= A (W,Z)$ with 
$$A:=\left(
\begin{array}{cc}
 \cot (\te_3) & 1 \\
 \cot (\te_3)-\cot (\te_2) & 0
\end{array}
\right).
$$
Then, using~\eqref{entropyinv}, it follows that~\eqref{sub1} is equivalent to the following subadditivity inequality: for every random variables $X$ and $Y$,
\begin{gather*}
\frac1{r'} \, S\big( \frac1{\cot(\te_3)-\cot(\te_2)} Y \big) + \frac1p\,  S\big( \sin(\te_2) (X-Y) \big)  +  \frac1q \, S\big(\sin(\te_3) X \big) \\
 \le S\big(X,Y\big)+ \log\big|\cot(\te_3)-\cot(\te_2)\big|.
\end{gather*}
This inequality is also equivalent, using again the scaling of entropy (in dimension $1$) to
\begin{equation}\label{sub2}
\frac1{r'} \, S\big(X\big) + \frac1p\,  S\big(X-Y)\big)  +  \frac1q \, S\big( Y \big)  \le S\big(X,Y\big) +D 
\end{equation}
with
\begin{eqnarray*}
D&=& \Big(1-\frac1{r'}\Big) \log\big|\cot(\te_3)-\cot(\te_2)\big| + \frac1{p}\log\sin(\te_2) +\frac1{q}\log\sin(\te_3).  \\
\end{eqnarray*}
Using that $\frac1{p'}+\frac1{q'}=\frac1{r'}$, it follows that
$ \cot(\te_2)-\cot(\te_3) = \frac1{r'} \sqrt{\frac{p'q'}r}$ and 
$$
D= -\frac1{r}\log\sqrt{r}+\frac1{r'}\log\sqrt{r'} + \frac1{p}\log\sqrt{p}-\frac1{p'}\log\sqrt{p'} + \frac1{q}\log\sqrt{q}-\frac1{q'}\log\sqrt{q'}.
$$
For $t>1$, set $C_t:=\sqrt{\frac{t^{1/t}}{t'^{1/t'}}}$
where as before $t'$ is the conjugate of $t$.
We have thus derived the following classical result.

\begin{theo}[Sharp Young's convolution inequality] Let $p,q,r>1$ satisfy~\eqref{coeffYoung}.
For every random variables $X,Y\in \R$,
$$\frac1{r'} \, S\big(X\big) + \frac1p\,  S\big(X-Y\big)  +  \frac1q \, S\big( Y \big)  \le S\big(X,Y\big) + \log\left(\frac{C_p \, C_q}{C_r}\right).$$
Furthermore, the inequality is sharp:  equality holds if and only if  $(X,Y)\in\R^2$ is a Gaussian vector whose covariance matrix is a multiple of $A^\ast A$.

Equivalently, for every $f\in L^p(\R)$ and $g\in L^q(\R)$,
\be\label{Young}
\|f\ast g\|_{L^r(\R)} \le \frac{C_p \, C_q}{C_r}\, \|f\|_{L^p(\R)} \, \|g\|_{L^q(\R)}.
\ee
\end{theo}

For the equality cases in the entropic inequality, note that
in view of Proposition~\ref{special}  equality  holds in~\eqref{sub1} if and only if $(W,Z)$ is a Gaussian vector with covariance a multiple of $\id_{\R^2}$, and $(X,Y)=A(W,Z)$. Next, note that Young's inequality~\eqref{Young} 
reduces to the case of nonnegative functions, which is then equivalent to following dual form of the entropic inequality: 
\begin{equation} \label{BLYoung}
\iint f(x) g(x-y) h(y) \, dx\, dy  \le \frac{C_p \, C_q}{C_r}\, \|f\|_{L^p(\R)} \, \|g\|_{L^q(\R)} \, \|g\|_{L^q(\R)}
\end{equation}
for every nonnegative functions $f,g,h:\R\to\R^+$. It is possible to deduce the equality cases in this inequality (and therefore in Young's convolution inequality) from the ones in the entropic inequality, as described in~\cite{CC} (we get some well chosen Gaussian functions). Actually, it is also possible to use that inequality~\eqref{BLYoung} is obtained by rescaling the functions (after the change of variables $(x,y)=A(w,z)$) in the Brascamp-Lieb inequality dual to~\eqref{sub1}, where equality holds if and only the functions are Gaussian with the same covariance.

The sharp Young convolution inequality~\eqref{Young}  was obtained independently by Beckner~\cite{Beckner} and Brascamp and Lieb~\cite{BL}. Their proofs rely on rearrangements of functions and tensorization arguments. Barthe~\cite{B1} gave a new (simpler) proof using a mass transportation argument.  One of the advantage of the geometric entropic approach used here is that it makes it possible to extend it to other contexts.

\medskip

\section{Shannon's inequality}

We continue to work with the Lebesgue measure. We now aim at reproducing the following classical result in Information Theory (see~\cite{DCT} for details).

\begin{theo}[Shannon's inequality]
Let $X$ and $Y$ be two \emph{independent} random variables. Then
\be\label{shannon}
S\left(\frac{X+Y}{\sqrt 2}\right) \le \frac{S(X) + S(Y)}2 \,  .
\ee
\end{theo}

It is well known that in the case $Y$ (say) is symmetric $Y\sim -Y$, then Shannon's inequality follows from the classical subadditivity of the entropy~\eqref{eq:classicalsub} since
$$ 2 S\left(\frac{X+Y}{\sqrt 2}\right)  = S\left(\frac{X+Y}{\sqrt 2}\right)  + S\left(\frac{X-Y}{\sqrt 2}\right)  \le S(X,Y) = S(X) + S(Y)$$
where the last equality expresses the independence of $X$ and $Y$. However this situation is misleading since in the general case Shannon's inequality seems to be different in nature than an inequality of subadditivity of entropy. One of the obstacle is that we would like to use a decomposition of the identity with the basis vectors $e_1$ and $e_2$ together with $\frac{e_1 +e_2}{\sqrt 2}$. But this is not possible. We shall instead approximate such a situation.

Recall that if $G$ stands for a standard Gaussian variable independent of all the variables considered here, then if $X$ has finite entropy, $S(X+ \var\,  G ) \to S(X)$ when $\var\to 0$. Therefore we can restrict our study to the case where $X$ and $Y$ have smooth densities (with sub-gaussian tails).  For such regular variables, it is well known that when $\var \to 0$,
\be\label{DLentropy}
S(X+\var\, Y) = S(X) + \textrm{O}(\var^2).
\ee
Using the notation $u(\theta)= (\cos(\theta), \sin(\theta))$, introduce for fixed
$s\in (-\frac{\pi}2, 0)$ ($s$ will later  tend to $0^-$) the unit vectors 
$$u_1 =u(s), \quad u_2 = u(\textstyle{\frac{\pi}4}), \quad u_3 = u(\textstyle{\frac{\pi}2} - s).$$
These vectors define three directions
satisfying the assumption of Proposition~\ref{theo:decomp}. Let $c_1, c_2, c_3\in (0,1)$
be the associated coefficients for which there is a decomposition of the identity so that, by Proposition~\ref{special},
\be\label{shannonsub}
c_1\, S\big(\cos(s) X + \sin(s) Y\big) + c_2\, S\big(\textstyle{\frac{X+Y}{\sqrt 2}}\big) + c_3 \, S\big(\sin(s) X + \cos(s) Y\big) \le S(X,Y)
\ee
for every random variables $X$ and $Y$ with finite entropy. By Proposition~\ref{theo:decomp}, as $s\to 0^-$,
\begin{gather*}
c_1 = \frac{\cos(s)}{\sin(\textstyle{\frac{\pi}4}-s)\sin(\textstyle{\frac{\pi}2}-2s)} = 1 + 2 s + o(s), \quad
c_2= \frac{\cos(\textstyle{\frac{\pi}2}-2s)}{\sin(\textstyle{\frac{\pi}4}-s)\sin(s-\textstyle{\frac{\pi}4})}
 = -4s+o(s), \\
c_3 = \frac{\cos(\textstyle{\frac{\pi}2}-s)}{\sin(2s-\textstyle{\frac{\pi}2})\sin(s)} =  1 + 2 s +o(s). 
\end{gather*}
Note also  that when $X$ and $Y$ are (regular enough) random variables and $s\to 0^-$ we have, in view of~\eqref{DLentropy}, that
$$S\big(\cos(s) X + \sin(s) Y\big) = S\Big(X + \frac{\sin(s)}{\cos(s)} Y\Big) - \log\cos(s) = S(X) + o(s)$$
and similarly $ S\big(\sin(s) X + \cos(s) Y\big) = S(Y) + o(s)$.
Therefore, making a Taylor expansion in~\eqref{shannonsub} when $X$ and $Y$ are independent
and $s\to 0^-$, it follows that
$$S(X) + S(Y) -s\Big[ 4 S\big(\textstyle{\frac{X+Y}{\sqrt 2}}\big) -2 S(X) - 2 S(Y) \Big]  + o(s) \le S(X,Y)= S(X)+ S(Y).$$
The first order in $s<0$  gives the desired Shannon inequality.

\begin{rem}
In view of the duality between entropy and Brascamp-Lieb inequalities, one could wonder if the Shannon inequality admits a dual form. However, if we start with the Brascamp-Lieb inequality, in the same situation as above, and perform the Taylor expansion there, then we end up again with the Shannon inequality. We are in a situation where the  entropic and Brascamp-Lieb inequalities coincide at the first order.
\end{rem}

It should be noted that one can prove along the same lines the Blachmann-Stam inequality (cf.~\cite{DCT}):
\be\label{BS}
I\left(\frac{X+Y}{\sqrt 2}\right) \le \frac{I(X) + I(Y)}2 
\ee
for independent random variables with finite Fisher information. Indeed, the decomposition of the identity obtained above and Proposition~\ref{special} and~\eqref{infosub} give the result once it has been noted that for regular enough random variables, as for entropy,
$I(X+\var\, Y) = I(X) + \textrm{O}(\var^2)$ as $\var\to 0$.
Moreover, for independent random variables $I(X,Y)=I(X)+I(Y)$. Note that by the scaling of information, inequality~\eqref{BS} is commonly rewritten as
$I\left(X+Y\right) \le I(X) + I(Y) .$

Finally, we would like to mention that we could as well have started from Theorem~\ref{maintheo}. A first order Taylor expansion when $p_2=p_3\to 2$ (and therefore $p_1\to 1$) in~\eqref{sub0} gives again the Shannon
inequality. Having derived the sharp Young inequality from~\eqref{sub0}, this procedure is reminiscent of Dembo's proof of Shannon's inequality which consisted in a Taylor expansion in~\eqref{Young} around $p=q=2$ (see~\cite{DCT}).

\medskip


\section{Hypercontractivity and logarithmic Sobolev inequality}

In this section, the measure $\mu=\gamma$ will be the Gaussian measure and $\mu_n=\gamma_n$. 

Assume we are given $p,q$ with
\be\label{hypcond1}
1<p < q
\ee
and set $\theta\in [0, \frac{\pi}{2})$ such that
\be \label{hypcond2}
\cos(\theta)= \sqrt{\frac{p-1}{q-1}} \,.
\ee
Write as before $t'=t/(t-1)$ for $t > 1$.  Since $q > p$, let $r\in [1,q)$ be such that
$\frac1p = \frac1q + \frac1{r'}$, or equivalently 
$$\frac1{q'} + \frac1p +  \frac1{r} =2.$$
Then, introducing $\xi \in (\frac{\pi}2, \pi)$ such that
\be\label{hypcond3} 
\cos(\xi)= -\sqrt{\frac{r-1}{q-1}} \, ,
\ee
the angles $\theta, \xi$ are exactly the ones associated to the  triple
$$p_1=q'\; ,\quad  p_2=p \; ,  \quad p_3=r$$
in Theorem~\ref{maintheo} ($\theta_1= \theta, \theta_2= \xi$).
Consequently, for every random variables $X,Y$,
\be\label{dualhyp1}
\frac1{q'} S_\gamma (X) + \frac1p S_\gamma\big(\cos(\theta) X + \sin(\theta) Y \big) + \frac1r S_\gamma\big(\cos(\xi) X+\sin(\xi)Y\big) \le S_{\gamma_2}(X,Y) .
\ee
We emphasize in the next proposition the corresponding convolution inequality~\eqref{int0}
of independent interest (as a stronger statement than the classical hypercontractivity).

\begin{prop}[Hypercontractivity]  \label{prop:hyp}
Let $p,r>1$. If $\frac1p + \frac1r = 1+ \frac1q$ and~\eqref{hypcond2}-\eqref{hypcond3} hold, then for every functions $f\in L^p(\gamma)$ and $ g\in L^r(\gamma)$,
$$ \left\|\int f\big(\cos(\theta) x + \sin(\theta) y \big)\; g\big(\cos(\xi) x + \sin(\xi) y \big) \;d\gamma(y)\right\|_{L^q(d\gamma(x))} \le \|f\|_{L^p(\gamma)}\, \|g\|_{L^{r}(\gamma)},
$$
with equality if and only if $f,g$ are of constant sign with $|f(x)|^p d\gamma(x)= K_1 \, e^{-\lambda |x-a_1|^2}dx$ and  $|g(x)|^r d\gamma(x)= K_2 \, e^{-\lambda |x-a_2|^2} dx$, $K_1,K_2, \lambda\ge 0$, $a_1,a_2\in \R$.
\end{prop}

Indeed, this result contains the hypercontractivity inequality for the Hermite semi-group 
$$\D P_\theta(f)(x):= \int f\big(\cos(\theta) x + \cos(\theta) y \big) \, d\gamma(y).$$ 
(One may work as well with  the Ornstein-Uhlenbeck semi-group 
$\overline{P}_t f := e^{tL}f = P_{\arccos (e^{-t})} f$). 
Proposition~\ref{prop:hyp} applied with $g\equiv 1$ namely indicates
that under~\eqref{hypcond1}-\eqref{hypcond2}, 
\be\label{hyp}
\forall f\in L^p(\gamma), \qquad  \|P_\te f\|_{L^q(\gamma)} \le  \|f\|_{L^p(\gamma)}.
\ee
Equality holds iff $f$ is exponential,  $f(x) = K e^{-a\, x}$ ($\lambda=1/2$ since $g=1$ in the Proposition).

If we rather work at the level entropic inequalities, first note that for every random variable $Z$, $S_{\gamma}(Z) \ge 0$, with equality if $Z$ is a standard Gaussian variable. Therefore, inequality~\eqref{dualhyp1} implies that, for every random variables $X,Y$,
\be\label{dualhyp2}
\frac1{q'} S_\gamma (X) + \frac1p S_\gamma\big((\cos(\theta) X + \sin(\theta) Y\big) \le S_{\gamma_2}(X,Y) ,
\ee
which is dual to the following Brascamp-Lieb inequality, equivalent to~\eqref{hyp}: for every functions $g,f:\R \to \R^+$, 
$$\iint g(x) \; f\big(\cos(\theta) x + \sin(\theta) y \big)\;  d\gamma(x)d\gamma(y) \le \|f\|_{L^p(\gamma)}\, \|g\|_{L^{q'}(\gamma)}.$$

\begin{rem}[Particularity of the Gaussian case]
It is worth noting that in the Gaussian case, inequalities~\eqref{entropysub2} and~\eqref{BL} hold under the weaker condition
\be\label{dec3}
\sum_{i=1}^m c_i \, u_i\otimes u_i \le {\rm Id}_{\R^n}.
\ee
(So here we can simply use that
$\frac1{q'} e_1\otimes e_1 + \frac1p u(\theta)\otimes u(\theta) \le {\rm Id}_{\R^2}$, instead of `forgetting' terms, as we did above.)
The reason is the following (this was also noted in the spherical case in~\cite{BCM} and it is in fact a general feature when working on a probability space). From the explanations given in the introduction, it is clear that~\eqref{dec3} is always sufficient to get~\eqref{infosub}. When integrating along the Heat semi-group,  it is necessary to rescale in order to obtain asymptotically a standard Gaussian. So the condition~\eqref{comp} is there crucial (see~\cite{CC}). But in the Gaussian case, there is no need to rescale when integrating along the Ornstein-Uhlenbeck semi-group, and so we can indeed derive inequalities~\eqref{entropysub2} and~\eqref{BL} from~\eqref{dec3}. 
Alternatively, starting from~\eqref{dec3}, we can complete the self-adjoint operator on the RHS in order to get a decomposition of the identity, by adding  some unit vectors  $u'_j$ and coefficients $c'_j$. We then simply apply~\eqref{entropysub2} and~\eqref{BL} in the case of this decomposition, but with $X_j= G$ (standard Gaussian independent of the rest) and $f_j\equiv 1$, respectively, for the added indices.
\end{rem}

It is well known that the Gaussian logarithmic Sobolev inequality is equivalent to the hypercontractivity inequality~\eqref{hyp} \cite{G}. To derive the logarithmic Sobolev inequality, one can differentiate~\eqref{hyp} at $\te=0$.
Let us explain how it is even easier to see this implication when working with the dual entropic form. 
Recall that for every random variable $Z$, $S_\gamma(Z)\ge 0$ with equality if $Z$ is a standard Gaussian variable. Let $X$ be a random variable and $G$ be a standard Gaussian variable independent of $X$. For $\theta\in [0,\frac{\pi}2)$, set 
$$P_\theta X:= \cos(\te)X + \sin(\te) G.$$
Then, inequality~\eqref{dualhyp2} gives that, under~\eqref{hypcond1}-\eqref{hypcond2},
$$\frac{1}{q'} S_\gamma(X) + \frac1p S_\gamma(P_\te X) \le  S_{\gamma_2}(X,G) = S_\gamma(X) $$
which rewrites as 
$\D S_\gamma(P_\te X)\le  \frac{p}q S_\gamma(X). $
Therefore, for any $\theta\in [0,\frac{\pi}2)$ and $q>1$, by picking the appropriate $p$ verifying~\eqref{dualhyp2},
$S(P_\te X)\le  \frac{1+(q-1)\cos^2(\te)}q S(X)$.
Letting $q\to +\infty$,
$$
S_\gamma(P_\te X ) \le \cos^2(\te)  S_\gamma(X)
$$
which is the well known  integrated form of the logarithmic Sobolev inequality. Indeed, since
there is equality $\te=0$, the $\theta^2$ order term gives the  logarithmic Sobolev inequality 
$$\D S_\gamma(X)\le\frac12 I_\gamma(X)=-\frac{d^2}{d\te^2}_{|\te=0}\, S_\gamma(P_\te X) .$$

\medskip


\section{Higher dimensional inequalities}

We have studied convolution inequalities for functions on $\R$ and for random variables. But the strategy applies word by word to convolution inequalities for functions on $\R^n$ and random vectors. Let us briefly explain why.

All subspaces of $\R^n$ are equipped with the Euclidean structure inherited from the standard Euclidean structure on $\R^n$. Accordingly, for a subspace $E\subset \R^n$, the measure $\mu_E$ will stand for the Lebesgue measure or the standard Gaussian measure on $E$. Denote by  $P_E$ the orthogonal projection in $\R^n$ onto $E$, and  if $f$ is a \pbd{\mu_n}, denote by $f_{(E)}$ the \pbd{\mu_E} which is the image of $fd\mu$ under the map $P_E$.
For every random vector $X\in \R^n$, we have that $X\sim f\, d\mu \Rightarrow P_E X \sim f_{(E)} \, d\mu_E$ and
$$\int_E g(y) f_{(E)}(y)\, d\mu_E(y) = \int_{\R^n} g(P_E x)  f(x) \, d\mu_n (x).$$
The following analogue of~\eqref{supadd1} immediately holds: for $f$ a (smooth)  \pbd{\mu_n}  (or any random vector with $X\sim f d\mu_n$),
\begin{equation}\label{infoproj}
I_{\mu_E} (f_{(E)})= I_{\mu_E} (P_E X) \le \int_{\R^n} \frac{|P_E \nabla f |^2}f  \, d\mu_n .
\end{equation}
Assume we are given a collection of subspaces $E_1, \ldots E_m\subset \R^n$ and of positive numbers $c_1, \ldots , c_m>0$ such that
\be\label{multdecom}
\sum_{i=1}^m c_i \, P_{E_i} = \id_{\R^n}.
\ee
Then, by~\eqref{infoproj}-\eqref{multdecom},
$\sum_{i=1}^m c_i\, I_{\mu_{E_i}}(P_{E_i} X) \le I_{\mu_n} (X)$.
After integration along appropriate semi-groups $P_t$ (noting again that
$(P_t f)_{(E_i)} = P_t (f_{E_i})$), we get the analogue of Proposition~\ref{special}:
$$\sum_{i=1}^m c_i\, S_{\mu_{E_i}}(P_{E_i} X) \le S_{\mu_n} (X) ,$$
and by duality, for $f_i : E_i \to \R^+$, $i=1, \ldots, m$, the classical (multidimensional) geometric Brascamp-Lieb inequality
$$\int_{\R^n} \prod_{i=1}^m f^{c_i}(P_{E_i} x) \, d\mu_n \le \prod_{i=1}^m \left(\int_{E_i} f_i \, d\mu_{E_i} \right)^{c_i} .$$

The convolution inequalities on $\R^n$ are obtained by using appropriate projections onto three $n$-dimensional subspaces of $\R^{2n}$. Given an angle $\theta\in [0, \pi]$, denote by $P_\theta$ the projection in $\R^{2n}$ obtained by tensorizing the projection on the direction $u(\theta)$ in $\R^2$:
$$P_\theta = 
 U_\te^\ast U_\te, \quad \textrm{with} \quad U_\te = 
 \begin{pmatrix}
 \cos(\te)  \id_{\R^n} & 
       \sin(\te)  \id_{\R^n} 
 \end{pmatrix}:\R^{2n}\to\R^n.
$$
Identifying $\R^{2n}$ with $\R^n\times \R^n$, $P_\te (x,y) = U_\te ^\ast (\cos(\te) x + \sin(\te) y)$  for $x,y\in \R^n$. The projection onto the first $\R^n$ is $P_0$: $P_0(x,y)=x$.
Note that the image subspace $E_\theta:= \textrm{Im}(P_\te)$ is normally parametrized by $\R^n$ as
$E_\te = \{U_\te^\ast z ; \ z\in \R^n\}$.
Therefore, for a random vector $(X,Y)\in \R^{2n}$, 
$$S_{\mu_{E_\te}} \big(P_{\te}(X,Y)\big) = S_{\mu_n}(\cos(\te) X + \sin(\te) Y).$$

Assume then we are given $p_1, p_2, p_3>1$ with $\frac1{p_1} +\frac1{p_2} + \frac1{p_3} = 2$, and let $\te_2, \te_3$ be the angles given by Theorem~\ref{maintheo}. These angles came from the decomposition of the identity of Proposition~\ref{theo:decomp2}, which extends to a decomposition of the identity of $\R^{2n}$:
$$\frac1{p_1} P_0 + \frac1{p_2} P_{\te_2}  + \frac1{p_3} P_{\te_3} = \id_{\R^{2n}}.$$
By the previous considerations, Theorem~\ref{maintheo} immediately extends to random vectors on $\R^n$ and functions on $\R^n$: for every random vectors $X, Y\in \R^n$, 
$$
\frac1{p_1} \, S_{\mu_n}\big(X\big) + \frac1{p_2}\,  S_{\mu_n}\big(\cos(\te_2) X + \sin(\te_2) Y\big)  +  \frac1{p_3} \, S_{\mu_n}\big(\cos(\te_3) X + \sin(\te_3) Y\big)  \le S_{\mu_{2n}}\big(X,Y\big).
$$
Similarly,
for every functions $f,g:\R^n\to \R$ with $g\in L^{p_2}(\mu_n)$ and $ h\in L^{p_3}(\mu_n)$,
$$
 \left\|\int  g(\cos(\te_2) x + \sin(\te_2) y)\, h(\cos(\te_3) x + \sin(\te_3) y) \, d\mu_n(y) \right\|_{L^{p'_1}(d\mu_n(x))} \le \|g\|_{L^{p_2}(\mu_n)}\, \|h\|_{L^{p_3}(\mu_n)}.
 $$
The cases of equality are also the same. Then, the multidimentional forms of Young's convolution inequality, Shannon's inequality, Hypercontractivity and logarithmic Sobolev inequality are obtained by the same computations we have performed previously.

Of course, it is also possible to derive these inequalities from the one-dimensional ones by standard tensorization techniques. But as pointed out earlier,  the geometric approach used here might prove useful in some other contexts.

\bigskip


\bigskip

{\small
\noindent
D. C.-E.: Institut de Math\'ematiques de Jussieu, Universit\'e Pierre et Marie Curie (Paris 6),
4 place Jussieu, 75252 Paris Cedex 05, France, cordero@math.jussieu.fr
\medskip

\noindent
M.L.: Institut de Math\'ematiques de Toulouse, Universit\'e de Toulouse, 31062 Toulouse, France, ledoux@math.univ-toulouse.fr
}

\end{document}